\documentclass[12pt]{amsart}

\usepackage[margin=1in]{geometry}

\usepackage{a4wide}
\usepackage{amsmath}
\usepackage{amssymb}
\usepackage{stmaryrd}
\usepackage{amsthm}
\usepackage{amscd}
\usepackage{enumerate}
\usepackage[all,cmtip]{xy}
\usepackage{graphics}

\usepackage{hyperref}
\hypersetup{
    colorlinks,
    citecolor=black,
    filecolor=black,
    linkcolor=black,
    urlcolor=black
}

\numberwithin{equation}{section}

\usepackage[normalem]{ulem}
\useunder{\uline}{\ulined}{}

\theoremstyle{plain}
\newtheorem{defn}{Definition}[section]
\newtheorem{thm}[defn]{Theorem}

\newtheorem{lem}[defn]{Lemma}

\newtheorem{cor}[defn]{Corollary}

\newtheorem{prop}[defn]{Proposition}

\newtheorem*{thma}{Theorem}

\theoremstyle{definition}
\newtheorem{ex}[defn]{Example}

\theoremstyle{remark}

\newcommand{\Hom}{\mathrm{Hom}}

\newcommand{\Aut}{\mathrm{Aut}}

\mathchardef\mhyphen="2D

\newcommand{\Z}{\mathbb{Z}}

\newcommand{\Q}{\mathbb{Q}}

\newcommand{\GL}{\mathrm{GL}}

\newcommand{\quot}[2]{{\raisebox{.1em}{$#1$\!}\left/\raisebox{-.2em}{$#2$}\right.}}

\newcommand{\lra}{\longrightarrow}

\newcommand{\ra}{\rightarrow}

\newcommand{\langl}{\begin{picture}(4.5,7)
\put(1.1,2.5){\rotatebox{60}{\line(1,0){5.5}}}
\put(1.1,2.5){\rotatebox{300}{\line(1,0){5.5}}}
\end{picture}}
\newcommand{\rangl}{\begin{picture}(4.5,7)
\put(.9,2.5){\rotatebox{120}{\line(1,0){5.5}}}
\put(.9,2.5){\rotatebox{240}{\line(1,0){5.5}}}
\end{picture}}

\newcommand\restr[2]{{
  \left.\kern-\nulldelimiterspace 
  #1 
  \vphantom{\big|} 
  \right|_{#2} 
  }}
  
\begin{document}
\title[Endomorphisms of nilpotent groups of finite rank]{Endomorphisms of nilpotent groups of finite rank}
\author{Hector Durham}
\address{Mathematical Sciences, University of Southampton, SO17 1BJ, UK}
\email{H.Durham@soton.ac.uk}
\begin{abstract}
We obtain sufficient criteria for endomorphisms of torsion-free nilpotent groups of finite rank to be automorphisms, by considering the induced maps on the torsion-free abelianisation and the centre.
Whilst these results are known in the finitely generated case removing this assumption introduces several difficulties.
\end{abstract}
\maketitle

\section{Introduction}

Suppose an endomorphism $\sigma$ of a torsion-free finitely generated nilpotent group $N$ induces an isomorphism on the centre.
It is known by (independent) work of Farkas \cite{Farkas} and Wehrfritz \cite{Wehrfritz I} that $\sigma$ is necessarily an automorphism.
Meanwhile Wehrfritz [ibid.] demonstrates that for $N$ of finite (Pr\"ufer) rank, still torsion-free and nilpotent, the result is false in general.
For $\pi$ a set of rational primes, we obtain sufficient criteria in the $\pi$-divisible case by considering a generalisation of integer-like endomorphisms to so-called $\pi$-like endomorphisms.
Invertible integer-like endomorphisms (the case $\pi=\emptyset$) are discussed in the context of nilpotent groups in \cite{Dekimpe Dere}.
The integer-like endomorphisms are those that preserve a maximal rank torsion-free finitely generated abelian group in the Lie algebra of the Mal'cev completion.
As a special case of theorem (\ref{thm: central main theorem}), we then have the following.
\begin{thma}
Let $N$ be torsion-free nilpotent of finite rank, and $\sigma$ an integer-like endomorphism with $\det\restr{\sigma}{Z(N)}=\pm1$.
Then $\sigma$ is an automorphism of $N$.
\end{thma}
We give counterexamples (see \ref{ex:counterexamples}) to show that one must have a condition on the determinant and furthermore that some version of integer-like must be considered.
Meanwhile Farkas and Wehrfritz, again independently, show that an endomorphism of a polycyclic-by-finite group which induces an isomorphism on the Zaleskii subgroup is an automorphism - see also the later paper of Wehrfritz \cite{Wehrfritz II} in this context.
We give an example (see \ref{ex:wehrfritz does not generalise}) to show that this cannot be generalised to finitely generated minimax groups.

On the other hand, we show (theorem (\ref{thm: tfab theorem})) that the torsion-free abelianisation always detects the surjectivity of an endomorphism, a more straightforward result.
This is the content of the first section of the paper, since the tools are required later.

\section{Detecting Surjectivity on the Torsion-Free Abelianisation}
The following proposition is standard, but we give here a precise formulation for our specific application later.
For any group $G$ and $i\geqslant1$, denote by $\gamma_i(G)$ the $i$-th term of its lower central series, and $\Gamma_i(G)$ its isolator, that is, the preimage of the torsion subgroup of $G/\gamma_i(G)$.
As usual we often write $G'=\gamma_2(G)$.
We say that an endomorphism $\sigma$ preserves $H$ if $\sigma(H)\leqslant H$ and stabilises $H$ if $\sigma(H)=H$. 

\begin{lem}[{\cite[1.2.11]{LR}}]\label{lem: tensor power maps}
Let $N$ be a torsion-free nilpotent group and $\sigma$ an endomorphism of $N$.
For each $i\geqslant1$, commutation in $N$ induces a $\Z[\sigma]$-module homomorphism 
\begin{align*}
\alpha_i\colon\big(\quot{N}{\Gamma_2}(N)\big)^{\otimes i}&\lra\quot{\Gamma_i(N)}{\Gamma_{i+1}}(N),\\
\bar x_1\otimes\cdots\otimes \bar x_i&\longmapsto\overline{[x_1,\dots,x_i]},
\end{align*}
where we equip the tensor power with the diagonal action.
Moreover, the image of $\alpha_i$ is precisely $$\quot{\gamma_i(N)\Gamma_{i+1}(N)}{\Gamma_{i+1}(N)}.$$
\end{lem}
\begin{proof}
We remark only that the formulation here is a standard generalisation of the cited result.
\end{proof}
We now prove here a certain rigidity result for torsion-free abelian groups of finite rank.
It is presumably well-known, but we are unaware of a reference.

We begin with the following elementary lemma.
Notation: for an abelian group $A$, and natural number $l\geqslant0$, set $A[l]$ to be the subgroup of $A$ consisting of elements of order dividing $l$.
\begin{lem}\label{lem:part1 abelian group lemma}
Let $p$ be a prime, and $T$ a $p$-torsion abelian group with $T[p]$ finite.
Then every injective endomorphism of $T$ is surjective.
\end{lem}
\begin{proof}
Let $\sigma$ be an injective endomorphism of $T$.
Since $T=\bigcup_{i\geqslant0}T[p^{i}]$ and these are fully invariant subgroups, it suffices to show that the restriction of $\sigma$ to $T[p^{i}]$ is surjective for each $i$.
Since $\sigma$ is injective, it will suffice to show that these subgroups are all finite.
The case $i=1$ is our hypothesis, so assume that $T[p^{i}]$ is finite and note that $T[p^{i+1}]$ is an extension of $T[p]$ by $pT[p^{i+1}]\leqslant T[p^{i}]$.
Conclude.
\end{proof}
We may now show
\begin{prop}\label{prop:part1 abelian group prop}
Let $A$ be a torsion-free abelian group of finite rank and $B$ a subgroup of the same rank.
Suppose $\sigma$ is an endomorphism of $A$ which stabilises $B$.
Then $\sigma\in\Aut(A)$.
\end{prop}
\begin{proof}
By the five lemma, it will suffice to show that the induced endomorphism on $A/B$ is surjective.
One verifies immediately from the necessary injectivity of $\sigma$ on $A$ that is also injective on this quotient.
Moreover, decomposing the torsion group $A/B$ into its primary components, we may assume that $A/B$ is $p$-torsion for some prime $p$.
Lemma (\ref{lem:part1 abelian group lemma}) now applies due to our finiteness of rank assumption.
\end{proof}
We are now in a position to prove
\begin{thm}\label{thm: tfab theorem}
Let $N$ be torsion-free nilpotent of finite Pr\"ufer rank, and $\sigma$ an endomorphism of $N$ with $\sigma(N)\Gamma_2(N)=N$.
Then $\sigma\in\Aut(N)$.
\end{thm}
\begin{proof}
Note firstly that $\sigma$ is necessarily injective by a Hirsch length argument. 
We induct on the class $c$ of $N$, the case $c=1$ being clear.
For $N$ of class $c>1$, a further Hirsch length argument shows that the induced map $\bar\sigma\colon N/\Gamma_c(N)\ra N/\Gamma_c(N)$ is injective.
By induction we deduce that $N=\sigma(N)\Gamma_c(N)$.

Consider now the $\Z[\sigma]$-module homomorphism $\alpha_c$ as in lemma (\ref{lem: tensor power maps}).
Since by hypothesis $\sigma$ acts as an automorphism on $N/\Gamma_2(N)$, the image $\gamma_c(N)$ of $\alpha_c$ is stabilised by $\sigma$.
Proposition (\ref{prop:part1 abelian group prop}) applies and we deduce that $\sigma(\Gamma_c(N))=\Gamma_c(N)$.
Finally $N=\sigma(N)\Gamma_c(N)=\sigma(N)\sigma(\Gamma_c(N))=\sigma(N)$, as desired.
\end{proof}
\section{Detecting Surjectivity via the Centre}
We introduce first the main tool required to obtain our second result.
For a group $G$ and $i\geqslant0$, denote the $i$-th term of the upper central series of $G$ by $Z^{i}(G)$.
We often simply write $Z^{1}(G)=Z(G)$.
\begin{lem}[{\cite[1.2.19]{LR}}]\label{lem:upper central hom map}
Let $N$ be a nilpotent group and $\omega$ an automorphism of $N$.
Then for each $i>0$ there is an $\omega$-equivariant split monomorphism of abelian groups
\begin{align*}
\beta_i\colon Z^{i+1}(N)/Z^{i}(N)&\lra\Hom\big(N/N',Z^{i}(N)/Z^{i-1}(N)\big)\\
\bar w&\longmapsto\big(\bar x\mapsto\overline{[w,x]}\big),
\end{align*}
where $\omega$ acts on the right hand side by sending some $\theta$ to the map $\theta^{\omega}$, which for $\bar x$ in $N/N'$ is defined by
$$\theta^{\omega}(\bar x)=\omega\theta\big(\overline{\omega^{-1}(x)}\big).$$
\end{lem}

Now let $N$ be a torsion-free nilpotent group of finite rank and $\sigma$ an injective endomorphism of $N$.
Furthermore denote by $R$ the Mal'cev completion of $N$ and $\sigma_\times$ the induced automorphism of $R$.
Note that the upper central series of $N$ is indeed preserved by $\sigma$, a consequence for example of the elementary fact that for each $i\geqslant0$, we have $Z^{i}(N)=Z^{i}(R)\cap N$, so that
\begin{equation*}
\sigma\big(Z^{i}(N)\big)=\sigma_\times\big( Z^{i}(R)\big)\cap \sigma(N)=Z^{i}(R)\cap \sigma(N)\leqslant Z^{i}(N).
\end{equation*}
Together with the map $\beta_1$ described in lemma (\ref{lem:upper central hom map}), the natural maps $Z(N)\ra Z(R)$ and $N/N'\ra R/R'$ induce maps which fit together as below.
\begin{equation}\label{eqn: the sequence}
\begin{gathered}
\xymatrixcolsep{2pc}
\xymatrixrowsep{2pc}
\xymatrix{
Z^{2}(N)/Z^{1}(N)\ar@{->}[d]^-{\beta_1}\\
\Hom(N/N',Z(N))\ar@{->}[r]^-{\gamma}&\Hom(N/N',Z(R))\ar@{<-}[d]^-{\delta}\\
&\Hom(R/R',Z(R))
}
\end{gathered}
\end{equation}
The required properties of this sequence are detailed in the following proposition.
\begin{prop}\label{prop:the sequence proposition}
Let $N$, $R$, $\sigma$, $\sigma_\times$ be as above and consider diagram (\ref{eqn: the sequence}).
We claim the following.
\begin{enumerate}
\item $\gamma$ is injective.
\item $\delta$ is an isomorphism.
\item $\delta^{-1}\circ\gamma\circ\beta_1$ is a $\Z[\sigma_\times]$-module map.
\item $\delta^{-1}\circ\gamma$ is a $\Z[\sigma_\times^{-1}]$-module map, provided that $\sigma(Z(N))=Z(N)$.
\end{enumerate}
\end{prop}
\begin{proof}
The first claim is an immediate consequence of the injectivity of $Z(N)\ra Z(R)$.

That $\delta$ is an isomorphism follows from the fact that $Z(R)$ is a $\Q$-vector space and the map $N/N'\ra R/R'$ is naturally isomorphic to tensoring with $\Q$.


For the third part, note that $\sigma_\times$ has a well defined action (as specified in lemma (\ref{lem:upper central hom map})) on $\Hom(R/R',Z(R))$ since it is an automorphism of $R$.
In order to show it is equivariant, let $wZ$ be an element of $Z^{2}(N)/Z^{1}(N)$.
It is required to show that we have an equality of maps
\begin{align*}
\left(\delta^{-1}\circ\gamma\circ\beta_1\right)(\sigma(w)Z)=\bigg(\left(\delta^{-1}\circ\gamma\circ\beta_1\right)(wZ)\bigg)^{\!\sigma_\times}\!\!\colon R/R'\lra Z(R).
\end{align*}
Since $\delta$ is an isomorphism it suffices to check that these are equal after precomposing with the natural map $N/N'\ra R/R'$.
It thus suffices to show that if $x\in N$ that
\begin{align*}
\bigg(\left(\delta^{-1}\circ\gamma\circ\beta_1\right)(\sigma(w)Z)\bigg)(xR')= \sigma_\times\Bigg(\bigg(\left(\delta^{-1}\circ\gamma\circ\beta_1\right)(wZ)\bigg)(\sigma_\times^{-1}(x)R')\Bigg).
\end{align*}
The left hand side now immediately reduces to $[\sigma(w),x]$, so that this equality holds if and only if
\begin{align}
\sigma_\times^{-1}[\sigma(w),x]=\bigg(\left(\delta^{-1}\circ\gamma\circ\beta_1\right)(wZ)\bigg)(\sigma_\times^{-1}(x)R').\label{eqn:upper central hideosity}
\end{align}
Note that $\sigma_\times^{-1}(x)$ may not lie in $N$, but there certainly exists some $l\geqslant1$ for which $\sigma_\times^{-1}(x)^{l}\in N$.
Taking the $l$-th multiple in $Z(R)$ of the right hand side of equation (\ref{eqn:upper central hideosity}) we thus see that
\begin{align*}
l\cdot\bigg(\left(\delta^{-1}\circ\gamma\circ\beta_1\right)(wZ)\bigg)(\sigma_\times^{-1}(x)R')&=\bigg(\left(\delta^{-1}\circ\gamma\circ\beta_1\right)(wZ)\bigg)(\sigma_\times^{-1}(x)^{l}R')\\
&=[w,\sigma_\times^{-1}(x)^{l}]\\
&=l\cdot[w,\sigma_\times^{-1}(x)].
\end{align*}
Appealing to the unique divisibility of $Z(R)$ we deduce that the right hand side of equation (\ref{eqn:upper central hideosity}) is precisely $[w,\sigma_\times^{-1}(x)]=\sigma_\times^{-1}[\sigma(w),x]$, as desired.

For the final part, note that the fact that the abelian group $\Hom(N/N',Z(N))$ is a $\Z[\sigma_\times^{-1}]$-module is where the additional hypothesis on $\sigma$ plays its role.
The proof of equivariance is both similar to and simpler than the one above, so we omit it.
\end{proof}

In all that follows, $\pi$ will denote a (possibly empty) set of prime numbers.
We begin by recalling the following standard notions.
A $\pi$-number is a rational integer with all prime divisors contained in $\pi$, and we notate $\Z[1/\pi]:=\{\frac{m}{n}\colon m\in\Z,n$ a non-zero $\pi$-number$\}$.
Finally a $\pi$-unit is a unit in this ring.
If $\pi=\{p\}$, we notate $\Z[1/\pi]=\Z[1/p]$ as usual.
We now introduce the notion of $\pi$-like morphisms.

\begin{defn}
Let $V$ be a rational vector space of dimension $n<\infty$, and let $\nu$ be an automorphism of $V$.
We say that $\nu$ is $\pi$-like if one of the following equivalent conditions hold.
\begin{itemize}
\item The coefficients of the characteristic polynomial of $\nu$ lie in $\Z[1/\pi]$.
\item There exists some $W\leqslant V$ preserved by $\nu$ with $W\cong\Z[1/\pi]^{n}$.
\end{itemize}
Now let $\sigma$ an injective endomorphism of a torsion-free nilpotent group $N$ of finite rank.
We say that $\sigma$ is $\pi$-like if the induced automorphism of the associated rational Lie algebra of $N$ is $\pi$-like.
\end{defn}
The equivalence of these two conditions is standard linear algebra.
In case $\pi=\emptyset$, this is the notion of integer-like automorphisms, as considered in \cite{Dekimpe Dere}.
This is a particularly well-behaved class of endomorphisms, as we see next.

\begin{prop}\label{prop:pi like is well behaved and easy to detect}
Let $N$ be torsion-free nilpotent of finite rank and $\sigma$ an injective endomorphism of $N$.
Then the following are equivalent.
\begin{enumerate}
\item $\sigma$ is $\pi$-like.
\item For any central series of $N$ preserved by $\sigma$ with torsion-free sections, the action of $\sigma$ on each section is $\pi$-like.
\item The induced map on $N/\Gamma_2(N)$ is $\pi$-like.
\end{enumerate}
\end{prop}
\begin{proof}
$1\implies2$: Mal'cev complete at the central series to obtain a decomposition of the associated rational Lie algebra $V$ as $0=V_0\leqslant V_1\leqslant\cdots\leqslant V_r=V$ with $\sigma_\times(V_i)=V_i$ for each $i$, where $\sigma_\times$ is the induced automorphism.
Let the characeristic polynomial of the induced automorphism on the section $V_{i}/V_{i-1}$ be $f_i$ and the characteristic polynomial of $\sigma_\times$ be $f$.
We obtain a factorisation $f=f_1\cdots f_n$.
By hypothesis $f$ has coefficients in $\Z[1/\pi]$, and Gauss' Lemma implies that each $f_i$ has coefficients in this ring too.

$2\implies 3$: Trivial.

$3\implies 1$: Select a subgroup isomorphic to a free $\Z[1/\pi]$-module of maximal rank in $N/\Gamma_2(N)$ preserved by $\sigma$.
The image of this subgroup under the tensor power maps described in lemma (\ref{lem: tensor power maps}) show that the action on each section of the isolated central series is $\pi$-like.
Upon Mal'cev completing, this factorises the characteristic polynomial of the induced automorphism $\sigma_\times$ into polynomials with coefficients in $\Z[1/\pi]$.
Conclude.
\end{proof}

A trivial consequence of the above proposition is the following, which we state separately for later clarity.
\begin{cor}\label{cor:gauss lemma}
Let $V$ be a finite dimensional $\Q$-vector space and $\omega$ a $\pi$-like automorphism of $V$.
Suppose that $U$ is a subspace of $V$ stabilised by $\omega$.
Then $\restr{\omega} U$ is $\pi$-like.
\end{cor}

%
We may now show the following.
\begin{thm}\label{thm: central main theorem}
Let $N$ be a $\pi$-divisible torsion-free nilpotent group of finite rank, and let $\sigma$ be a $\pi$-like endomorphism of $N$ such that $\det\restr{\sigma}{Z(N)}$ is a $\pi$-unit.
Then $\sigma\in\Aut(N)$.
\end{thm}
\begin{proof}
The proof is by induction on the class $c$ of $N$.
If $c=1$ and $N$ has rank $n$, select $W\cong\Z[1/\pi]^{n}$ preserved by $\sigma$ and $W\leqslant A$.
Then the hypothesis on the determinant implies that $\sigma(W)=W$, and we may conclude with proposition (\ref{prop:part1 abelian group prop}).

Thus suppose $c>1$.
By considering $N/Z(N)$ of smaller class and centre $Z^{2}(N)/Z^{1}(N)$, it will suffice to show that the determinant of the induced map on $Z^{2}(N)/Z^{1}(N)$ is also a $\pi$-unit.
In order to proceed, we will show that the action of $\sigma_\times^{-1}$ on $\Hom(R/R',Z(R))$ is $\pi$-like.
This will follow from the final part of proposition (\ref{prop:the sequence proposition}) once we know that the action of $\sigma_\times^{-1}$ on $\Hom(N/N',Z(N))$ is $\pi$-like.
Since $Z(N)$ is torsion-free, it is equivalent to show this for $\Hom(N/\Gamma_2(N),Z(N))$.
Our hypothesis implies, by proposition (\ref{prop:pi like is well behaved and easy to detect}), that the action of $\sigma$ on $N/\Gamma_2(N)$ is $\pi$-like.
Moreover by our determinant hypothesis we may deduce that the action of $\sigma^{-1}$ on $Z(N)$ is $\pi$-like.
Thus there are subgroups $S,T$ of $N/\Gamma_2(N)$ and $Z(N)$ respectively, both of maximal Hirsch length and isomorphic to direct sums of copies of $\Z[1/\pi]$, with $\sigma(S)\leqslant S$ and $\sigma^{-1}(T)\leqslant T$.
Consider the subgroup $H_{S,T}\leqslant \Hom(N/\Gamma_2(N),Z(N))$, consisting by definition of those $f$ for which $f(S)\leqslant T$.
One may conclude by observing that $H_{S,T}$ is preserved by $\sigma^{-1}$, and that $H_{S,T}$ is of maximal Hirsch length and also isomorphic to a direct sum of copies of $\Z[1/\pi]$.

Now let $U$ be the $\Q$-span of the image of $Z^{2}(N)/Z^{1}(N)$ in $\Hom(R/R',Z(R))$ under $\delta^{-1}\circ\gamma\circ\beta_1$, in the notation of proposition (\ref{prop:the sequence proposition}).
Then in particular $\sigma_\times^{-1}(U)=U$ and corollary (\ref{cor:gauss lemma}) applies with $\omega=\sigma_\times^{-1}$: denoting the map $\sigma_\times$ induces on $U$ by $\bar\sigma_\times$, we deduce that $(\det\bar\sigma_\times)^{-1}$ is a $\pi$-number.
Moreover $\det\bar\sigma_\times$ is a $\pi$-number by proposition (\ref{prop:pi like is well behaved and easy to detect}).
Thus it is a $\pi$-unit, as desired.
\end{proof}

We now justify why we prove theorem (\ref{thm: central main theorem}) under the stated hypotheses.
\begin{ex}\label{ex:counterexamples}
We consider the nilpotent group
\begin{align*}
N:=
\begin{pmatrix}
1&\Z&\Z[1/2]\\
&1&\Z[1/2]\\
&&1
\end{pmatrix}\leqslant\GL_3(\Q),
&
\quad
Z(N)=
\begin{pmatrix}
1&0&\Z[1/2]\\
&1&0\\
&&1
\end{pmatrix},
\end{align*}
with endomorphisms
\begin{align*}
\varphi_1
\begin{pmatrix}
1&a&c\\
&1&b\\
&&1
\end{pmatrix}
=
\begin{pmatrix}
1&2a&c\\
&1&b/2\\
&&1
\end{pmatrix},
&\quad
\varphi_2
\begin{pmatrix}
1&a&c\\
&1&b\\
&&1
\end{pmatrix}
=
\begin{pmatrix}
1&2a&2c\\
&1&b\\
&&1
\end{pmatrix}.
\end{align*}
Both $\varphi_1$ and $\varphi_2$ are injective endomorphisms of $N$ which are not surjective, but induce isomorphisms on the centre.

The first example demonstrates that we must assume that the induced map on the torsion-free abelianisation is $\pi$-like.
(Note that $\pi=\emptyset$ here.)
However proposition (\ref{prop:pi like is well behaved and easy to detect}) shows that this already implies that the whole endomorphism is $\pi$-like, whence this hypothesis.
Meanwhile, the second endomorphism is $\pi$-like but the determinant on the centre is not a $\pi$-unit, whence our second hypothesis.
\end{ex}

It is shown independently in \cite{Farkas} and \cite{Wehrfritz I} that an endomorphism of a polycyclic-by-finite group which restricts to an isomorphism of the Zaleskii subgroup is an automorphism.
We show finally that this cannot hold in the finitely generated minimax setting.

\begin{ex}\label{ex:wehrfritz does not generalise}
Let $N$ be as above and set
\begin{align*}
t:=
\begin{pmatrix}
1/2&&\\
&1&\\
&&1
\end{pmatrix},
&\quad
x:=
\begin{pmatrix}
1&&\\
&1&\\
&&2
\end{pmatrix}.
\end{align*}
Then $G:=\langl N,x\rangl$ is finitely generated minimax with Fitting subgroup $N$. Conjugating by $t$ gives a proper inclusion $G^{t}<G$ and moreover induces $\varphi_2$ on $N$ above.
The Zaleskii subgroup here is precisely the centre of the Fitting subgroup, where our endomorphism induces an isomorphism.
\end{ex}

\end{document}